\documentclass[11pt,a4paper]{article}
\usepackage{epsf,epsfig,amsfonts,amsgen,amsmath,amstext,amsbsy,amsopn,amsthm}
\usepackage{amsmath}
\usepackage{enumerate}
\usepackage{hhline}
\usepackage{multirow}
\usepackage{multicol}
\usepackage{booktabs}
\usepackage{lineno}
\usepackage{amsfonts,amsthm,amssymb,bm}
\usepackage{amsfonts}
\usepackage{graphics}
\usepackage{latexsym,bm}
\usepackage{amsfonts,amsthm,amssymb,bbding}
\usepackage{indentfirst}
\usepackage{graphicx}
\usepackage{color}
\usepackage[colorlinks=true,anchorcolor=blue,filecolor=blue,linkcolor=red,urlcolor=blue,citecolor=blue]{hyperref}
\usepackage{float}
\usepackage{tikz,enumerate}
\usetikzlibrary{calc}
\usepackage{geometry}
\newcommand{\mb}[1]{\boldsymbol{#1}}
\newcommand{\n}{\noindent}

\newtheorem{thm}{Theorem}[section]
\newtheorem{conj}[thm]{Conjecture}

\newtheorem{pro}[thm]{Problem}

\newtheorem{cor}[thm]{Corollary}

\newtheorem{prop}[thm]{Proposition}
\newtheorem{rem}{Remark}[section]

\renewcommand{\le}{\leqslant}

\renewcommand{\ge}{\geqslant}

\begin{document}

\title{Counterexamples to a conjecture on graph inertia}

\author{
Hongzhang Chen\thanks{School of Mathematics and Statistics, Gansu
Center for Applied Mathematics, Lanzhou University, Lanzhou, Gansu,
730000, China. Email: \url{mnhzchern@gmail.com}.} \and Jianxi
Li\thanks{Corresponding author. School of Mathematics and
Statistics, Minnan Normal University, Zhangzhou, Fujian, 363000,
China. Email: \url{ptjxli@hotmail.com}. Partially supported by the
NSF of Fujian Province (No.\ 2026J001968).} }

\date{\today}
\maketitle

% \linenumbers \pagewiselinenumbers

% \modulolinenumbers[2]

\begin{abstract}

The inertia of a graph $G$ is
\(\operatorname{In}(G)=(n^+(G),n^0(G),n^-(G))\), where $n^+(G),\,
n^0(G),\, n^-(G)$ are the numbers of positive, zero and negative
eigenvalues of the adjacency matrix of $G$, respectively, counted
with multiplicities. Akbari, Elphick, Kumar, Pragada and Tang
[\emph{Discrete Math. 349 (2026) 114953}] conjectured that every
graph \(G\) satisfies
\[
        2n^+(G)\le n^-(G)(n^-(G)+1).
\]
In this note, we construct a family of reduced graphs
\(\{W_k:k\ge5\}\) with
\[
        \operatorname{In}(W_k)
        =
        \left(\binom{k}{2}+1,\ 0,\ k-1\right),
\]
each of which violates the conjectured inequality.
We also observe that deleting the vertex \(a_1\) from \(W_5\) gives a reduced graph with inertia \((10,0,4)\), answering a question raised in the same paper. The family also refutes a weaker inequality proposed there.
\end{abstract}



\section{Introduction}

All graphs considered here are finite and simple. Let $G$ be a graph
with vertex set $V(G)$ and edge set $E(G)$. We denote by $n$ and
$e(G)$ the \emph{order} and the \emph{size} of $G$, respectively.
For any vertex $u\in V(G)$, let $N_{G}(u)$ (or $N(u)$ for short) be the set of neighbors of \(u\). 
We write $\mb{1}$ for the all-ones column vector and $J=\mb{1}\mb{1}^{\mathsf{T}}$ for the all-ones matrix. Let $I$ be the identity matrix. Let $A(G)$ be
the adjacency matrix of a graph $G$ of order $n$. We write $n^+(G),\, n^0(G),\, n^-(G)$ for the
numbers of positive, zero and negative eigenvalues of \(A(G)\),
respectively, counted with multiplicities. The inertia of \(G\) is
\[
\operatorname{In}(G)=(n^+(G),n^0(G),n^-(G)).
\]
The rank and the signature of \(G\) are
\[
        \operatorname{rank}(G)=n^+(G)+n^-(G),
        \qquad
        s(G)=n^+(G)-n^-(G).
\]
Two vertices \(u,v\in V(G)\) are twins if they have the same open neighbourhood, that is, \(N(u)=N(v)\). A graph is reduced if it has no isolated vertices and no twins.

The distribution of the adjacency eigenvalues on the two sides of the origin is a classical theme in spectral graph theory; see, for instance, Cvetkovi\'c, Doob and Sachs~\cite{CvetkovicDoobSachs} and Brouwer and Haemers~\cite{BrouwerHaemers}. Besides the nullity and the rank, the two quantities \(n^+(G)\) and \(n^-(G)\) are involved in problems on graph rank, graph signature, energy, and extremal questions with prescribed spectral restrictions. The problem of bounding the order of a graph in terms of its rank was studied by Kotlov and Lov\'asz~\cite{KotlovLovasz}, and related questions on ranks and signatures of adjacency matrices were considered by Akbari, Cameron and Khosrovshahi~\cite{AkbariCameronKhosrovshahi}. Bounds for the positive and negative inertia indices of graphs have also been studied by Ma, Yang and Li~\cite{MaYangLi} and by Fan and Wang~\cite{FanWang}.

A basic obstruction is that isolated vertices and twins do not
change the positive and negative inertia in an essential way. Hence
extremal questions of this kind are usually meaningful only for
reduced graphs. Torga\v{s}ev~\cite{Torgasev1985b} proved
that, for every fixed integer \(k\), there are only finitely many reduced graphs $G$ with \(n^-(G)=k\). He~\cite{Torgasev1985a,Torgasev1989,Torgasev1992} also determined
the maximum possible values of \(n^+(G)\) for reduced graphs with
\(n^-(G)=1,2,3\); these values are \(1,3,6\), respectively. Thus the first unsettled case in this direction is \(n^-(G)=4\).

Another source of motivation comes from strongly regular graphs. Let \(G\) be
a primitive strongly regular graph with spectrum
\[
        d^{(1)},\ r^{(f)},\ s^{(g)},\qquad d>r>0>s.
\]
The classical absolute bound gives
\[
        2n\le g(g+3),
\]
as a consequence of the work of Delsarte, Goethals and Seidel~\cite{DelsarteGoethalsSeidel}; see also Seidel's survey~\cite{Seidel} and Brouwer and Haemers \cite{BrouwerHaemers}. In this case, \(n^+(G)=f+1\), \(n^-(G)=g\), and
\(n=n^+(G)+n^-(G)\). Thus this bound is equivalent to
\[
        2n^+(G)\le n^-(G)(n^-(G)+1).
\]

Motivated by this observation, Akbari, Elphick, Kumar, Pragada and
Tang~\cite{AkbariEtAl} proposed the following conjecture.

\begin{conj}[\cite{AkbariEtAl}]\label{c1-1}
For every graph \(G\),
\begin{equation}\label{e-1}
        2n^+(G)\le n^-(G)(n^-(G)+1).
\end{equation}
\end{conj}

Conjecture~\ref{c1-1} would extend the absolute bound from strongly
regular graphs to arbitrary graphs, and at the same time would give
a quadratic upper bound for \(n^+(G)\) in terms of \(n^-(G)\).
Akbari \emph{et al.} verified Conjecture~\ref{c1-1} for several
graph classes, including planar graphs, line graphs,
self-complementary graphs, graphs with few odd cycles, tensor
products, joins, and cographs \cite{AkbariEtAl}.
They raised the following problem in the same article.
\begin{pro}[\cite{AkbariEtAl}]\label{p-1}
Does there exist a reduced graph $G$ with
\begin{equation}\label{e-2}
        n^-(G)=4,\qquad n^+(G)=10?
\end{equation}
\end{pro}
At the same time, they also proposed the following conjecture.
\begin{conj}[\cite{AkbariEtAl}]\label{conj-2}
For any graph $G$ of order $n$, we have
\begin{equation}\label{e-3}
        2n\le (n-n^+(G))(n-n^+(G)+3).
\end{equation}
\end{conj}

In this note, we show that (\ref{e-1}) is false by
constructing a family of reduced graphs \(W_k\) \((k\ge 5)\)
with
\[
        \operatorname{In}(W_k)=\left(\binom{k}{2}+1,\ 0,\ k-1\right).
\]
That is,
\[
        2n^+(W_k)
        =
        n^-(W_k)\bigl(n^-(W_k)+1\bigr)+2.
\]
Thus each \(W_k\) violates (\ref{e-1}). The smallest member \(W_5\) has inertia
$\operatorname{In}(W_5)=(11,0,4)$. Moreover, deleting the vertex \(a_1\) from \(W_5\) gives a reduced graph with inertia
$\operatorname{In}(W_5-a_1)=(10,0,4)$, which answers (\ref{e-2})
affirmatively. The same family also disproves the weaker inequality
(\ref{e-3}).

\section{The graphs \texorpdfstring{\(W_k\)}{Wk}}

Let \(k\ge 5\), and write $[k]=\{1,\ldots,k\}$.
Define a graph \(W_k\) as follows. Its vertex set is
\[
        V(W_k)=\{a_i:i\in [k]\}\cup\left\{b_S:S\in \binom{[k]}2\right\}.
\]
The vertices \(a_1,\ldots,a_k\) induce a complete graph \(K_k\). The vertices
\(b_S\) induce the Kneser graph on the 2-subsets of \([k]\): thus \(b_S\) is
adjacent to \(b_T\) if and only if \(S\cap T=\varnothing\). Finally, \(a_i\) is
adjacent to \(b_S\) if and only if \(i\in S\). Figure~\ref{fig:W5} is an example of $W_5$.

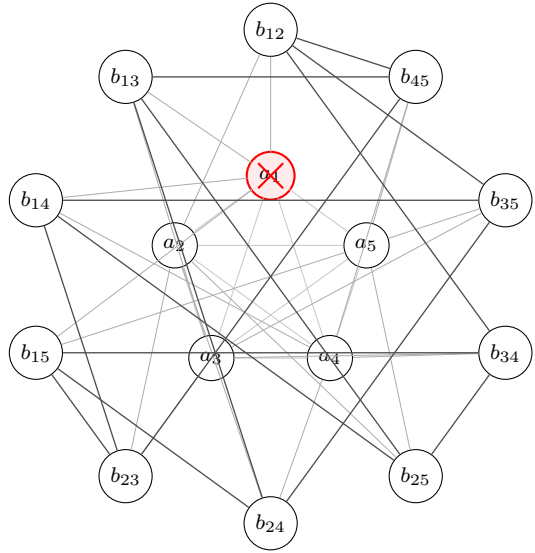
\begin{figure}[H]
\centering
\begin{tikzpicture}[scale=0.92,
  core/.style={circle,draw,fill=white,inner sep=1.3pt,minimum size=17pt,font=\scriptsize},
  outer/.style={circle,draw,fill=white,inner sep=1pt,minimum size=20pt,font=\scriptsize},
  rem/.style={circle,draw=red,thick,fill=red!8,inner sep=1.3pt,minimum size=18pt,font=\scriptsize}]

% core vertices a_i
\foreach \i/\ang in {1/90,2/162,3/234,4/306,5/18}{
  \node[core] (a\i) at (\ang:1.45) {$a_{\i}$};
}
\node[rem] at (a1) {$a_1$};
\draw[red,thick] ($(a1)+(-0.18,-0.18)$)--($(a1)+(0.18,0.18)$);
\draw[red,thick] ($(a1)+(-0.18,0.18)$)--($(a1)+(0.18,-0.18)$);

% b vertices on an outer decagon
\foreach \name/\lab/\ang in {
  b12/{12}/90,
  b13/{13}/126,
  b14/{14}/162,
  b15/{15}/198,
  b23/{23}/234,
  b24/{24}/270,
  b25/{25}/306,
  b34/{34}/342,
  b35/{35}/18,
  b45/{45}/54}{
  \node[outer] (\name) at (\ang:3.55) {$b_{\lab}$};
}

% Edges inside K_5
\foreach \i/\j in {1/2,1/3,1/4,1/5,2/3,2/4,2/5,3/4,3/5,4/5}{
  \draw[gray!45,line width=0.35pt] (a\i)--(a\j);
}

% Edges between a_i and b_S when i in S
\foreach \i/\s in {1/b12,2/b12,1/b13,3/b13,1/b14,4/b14,1/b15,5/b15,
                   2/b23,3/b23,2/b24,4/b24,2/b25,5/b25,
                   3/b34,4/b34,3/b35,5/b35,4/b45,5/b45}{
  \draw[gray!65,line width=0.35pt] (a\i)--(\s);
}

% Kneser edges among b_S vertices: disjoint 2-subsets
\foreach \u/\v in {b12/b34,b12/b35,b12/b45,b13/b24,b13/b25,b13/b45,
                   b14/b23,b14/b25,b14/b35,b15/b23,b15/b24,b15/b34,
                   b23/b45,b24/b35,b25/b34}{
  \draw[black!70,line width=0.5pt] (\u)--(\v);
}

\end{tikzpicture}
\caption{The graph \(W_5\).}
\label{fig:W5}
\end{figure}

Set $N=\binom{k}{2}$.
Let \(C\) be the \(k\times N\) point--2-subset incidence matrix, namely
\[
        C_{i,S}=
        \begin{cases}
        1,& i\in S,\\
        0,& i\notin S.
        \end{cases}
\]

Let \(K\) be the adjacency matrix of the Kneser graph on \(\binom{[k]}2\).
With the above ordering of the vertices,
\begin{equation}\label{e-4}
        A(W_k)=
        \begin{pmatrix}
        J_k-I_k & C\\
        C^T & K
        \end{pmatrix}.
\end{equation}

We shall use the identities
\begin{equation}\label{e-5}
        CC^T=(k-2)I_k+J_k
\end{equation}
and
\begin{equation}\label{e-6}
        K=J_N+I_N-C^TC.
\end{equation}
Indeed, each point of \([k]\) lies in \(k-1\) two-subsets, while two distinct
points lie together in exactly one two-subset, which gives (\ref{e-5}). Also
\((C^TC)_{S,T}=|S\cap T|\), so the right hand side of
(\ref{e-6}) has entry \(1\)
exactly when \(S\cap T=\varnothing\), and entry \(0\) otherwise, including on
the diagonal.

\begin{prop}\label{p2-1}
For \(k\ge 5\), the characteristic polynomial of \(W_k\) is
\begin{equation}\label{e-7}
\begin{aligned}
        \chi_{W_k}(x)
        &=
        (x-1)^{N-k}
        \bigl[x^2+(k-2)x-1\bigr]^{k-1}       \\
        &\quad\cdot
        \left[
        x^2-\left(k-1+\binom{k-2}{2}\right)x+(k-1)\left(\binom{k-2}{2}-2\right)
        \right].
\end{aligned}
\end{equation}
Consequently,
\[
        \operatorname{In}(W_k)=\left(\binom{k}{2}+1,\ 0,\ k-1\right).
\]
\end{prop}

\begin{proof}
The space
\[
        \mathbb R^k\oplus\mathbb R^N
\]
decomposes into three \(A(W_k)\)-invariant parts.

Let \(U=\mb{1}_k^\perp\). Since \(C^T\mb{1}_k=2\mb{1}_N\) and
\(CC^T=(k-2)I_k+J_k\) is nonsingular, the matrix \(C\) has rank \(k\).
Moreover, \(C^TU\) is orthogonal to \(\mb{1}_N\), and \(C^T\) is injective
on \(U\). Indeed, if \(C^T\mb{x}=\mb{0}\) with \(\mb{x}\in U\), then
\[
        CC^T\mb{x}=(k-2)\mb{x}=\mb{0},
\]
and hence \(\mb{x}=\mb{0}\). Thus \(\dim C^TU=k-1\). Also, if
\(\mb{y}\in\ker C\), then
\[
        2\mb{1}_N^T\mb{y}=\mb{1}_k^TC\mb{y}=0,
\]
so \(\ker C\subseteq \mb{1}_N^\perp\). Finally,
\[
        C^TU\cap \ker C=\{\mb{0}\}.
\]
Indeed, if \(C^T\mb{x}\in\ker C\) with \(\mb{x}\in U\), then
\[
        CC^T\mb{x}=(k-2)\mb{x}=\mb{0},
\]
and hence \(\mb{x}=\mb{0}\). Since
\[
        1+(k-1)+(N-k)=N,
\]
we obtain
\[
        \mathbb R^N
        =
        \langle \mb{1}_N\rangle
        \oplus C^TU
        \oplus \ker C.
\]
Therefore
\[
\begin{aligned}
        \mathbb R^k\oplus\mathbb R^N
        &=
        \langle(\mb{1}_k,\mb{0}),( \mb{0},\mb{1}_N)\rangle  \\
        &\quad\oplus
        \mathcal U_k
        \oplus
        (\{\mb{0}\}\oplus\ker C),
\end{aligned}
\]
where
\[
        \mathcal U_k=\{(\mb{x},C^T\mb{z}):\mb{x},\mb{z}\in U\}.
\]

We first consider the two-dimensional space spanned by
\[
        (\mb{1}_k,\mb{0}),\qquad (\mb{0},\mb{1}_N).
\]
Since each \(a_i\) is adjacent to \(k-1\) vertices of type \(a\) and \(k-1\)
vertices of type \(b\), while each \(b_S\) is adjacent to two vertices of type
\(a\) and \(\binom{k-2}{2}\) vertices of type \(b\), the action on this space
is represented by
\[
        Q_k=
        \begin{pmatrix}
        k-1 & k-1\\
        2 & \binom{k-2}{2}
        \end{pmatrix}.
\]
Thus this part contributes
\begin{equation}\label{e-8}
        x^2-\left(k-1+\binom{k-2}{2}\right)x
        +(k-1)\left(\binom{k-2}{2}-2\right).
\end{equation}

Since this quotient action is the restriction of the real symmetric matrix \(A(W_k)\) to an invariant subspace, the two roots of (\ref{e-8}) are real. For \(k\ge5\), the trace and determinant of \(Q_k\) are both positive; hence both roots of (\ref{e-8}) are positive.

Next we consider the subspace \(\mathcal U_k\). The map \(C^T\) is injective on
\(U\): if \(\mb{z}\in U\) and \(C^T\mb{z}=\mb{0}\), then
\(CC^T\mb{z}=(k-2)\mb{z}=\mb{0}\), and hence \(\mb{z}=\mb{0}\). Therefore
\(\mathcal U_k\) has dimension \(2(k-1)\). We show that \(\mathcal U_k\)
is \(A(W_k)\)-invariant. Let \(\mb{x},\mb{z}\in U\). By (\ref{e-5}),
\[
        CC^T\mb{z}=(k-2)\mb{z}.
\]
Moreover, since \(C\mb{1}_N=(k-1)\mb{1}_k\), we have
\[
        \mb{1}_N^T C^T\mb{z}=(C\mb{1}_N)^T\mb{z}=0.
\]
Thus \(C^T\mb{z}\perp \mb{1}_N\), and hence \(J_NC^T\mb{z}=\mb{0}\). By
(\ref{e-6}),
\[
        KC^T\mb{z}
        =
        (J_N+I_N-C^TC)C^T\mb{z}
        =
        C^T\mb{z}-C^TCC^T\mb{z}
        =
        (3-k)C^T\mb{z}.
\]
Therefore
\[
\begin{aligned}
A(W_k)(\mb{x},C^T\mb{z})
&=
\bigl((J_k-I_k)\mb{x}+CC^T\mb{z},\,
      C^T\mb{x}+KC^T\mb{z}\bigr)        \\
&=
\bigl(-\mb{x}+(k-2)\mb{z},\,
      C^T\mb{x}+(3-k)C^T\mb{z}\bigr)        \\
&=
\bigl(-\mb{x}+(k-2)\mb{z},\,
      C^T(\mb{x}+(3-k)\mb{z})\bigr).
\end{aligned}
\]
Thus \(\mathcal U_k\) is invariant. Under the identification
\[
        (\mb{x},C^T\mb{z})\longleftrightarrow(\mb{x},\mb{z}),
\]
the induced action is represented by
\[
        M_k=
        \begin{pmatrix}
        -1&k-2\\
        1&3-k
        \end{pmatrix}.
\]
The characteristic polynomial of \(M_k\) is
\begin{equation}\label{e-9}
        x^2+(k-2)x-1.
\end{equation}
Since \(\dim U=k-1\), the factor (\ref{e-9}) occurs with multiplicity
\(k-1\). It has one positive and one negative root.

Finally, if \(\mb{y}\in\ker C\), then \(C\mb{y}=\mb{0}\). Also
\[
        2\mb{1}_N^T \mb{y}=\mb{1}_k^TC \mb{y}=0,
\]
so \(J_N\mb{y}=\mb{0}\). By (\ref{e-6}),

\[
        K\mb{y}=\mb{y}.
\]
Since \((k-2)I_k+J_k\) is nonsingular, \(C\) has rank \(k\), and therefore
\[
        \dim\ker C=N-k.
\]
Thus \(\ker C\) contributes the factor \((x-1)^{N-k}\).

The three invariant parts have total dimension
\[
        2+2(k-1)+(N-k)=N+k=|V(W_k)|.
\]
Multiplying their characteristic factors gives (\ref{e-7}). Counting signs of the
roots gives
\[
        n^+(W_k)=N+1,\qquad n^0(W_k)=0,\qquad n^-(W_k)=k-1,
\]
as desired.
\end{proof}

\begin{prop}\label{p2-2}
For every \(k\ge5\), the graph \(W_k\) is reduced.
\end{prop}

\begin{proof}
The graph has no isolated vertices: the degree of \(a_i\) is \(2(k-1)\), and the degree of \(b_S\) is $2+\binom{k-2}{2}$. It remains to exclude twins. This can be read directly from the rows of \(A(W_k)\). With respect to the decomposition
\[
        \mathbb R^k\oplus \mathbb R^N,
\]
the part of the row corresponding to \(a_i\) in the
\(a\)-coordinates is $\mb{1}_k-\mb{e}_i$, where \(\mb{e}_i\) is the \(i\)-th standard basis vector of \(\mathbb R^k\), while the part of the row corresponding to \(b_S\)
in the \(a\)-coordinates is the incidence vector \(\mb{1}_S\) of
the $2$-subset \(S\).

If \(i\ne j\), then \(\mb{1}_k-\mb{e}_i\ne \mb{1}_k-
\mb{e}_j\), so \(a_i\) and \(a_j\) are not twins. If \(S\ne T\), then
\(\mb{1}_S\ne \mb{1}_T\), so \(b_S\) and \(b_T\) are not
twins. Finally,
\[
        |\operatorname{supp}(\mb{1}_k-\mb{e}_i)|=k-1,\qquad
        |\operatorname{supp}(\mb{1}_S)|=2,
\]
and these numbers are different for \(k\ge5\). Hence no \(a_i\) is a twin of
any \(b_S\). Therefore \(W_k\) is reduced.
\end{proof}

\begin{thm}
For every \(k\ge5\), the graph \(W_k\) is a reduced counterexample to
\[
        2n^+(G)\le n^-(G)(n^-(G)+1).
\]
\end{thm}

\begin{proof}
From Propositions~\ref{p2-1} and~\ref{p2-2}, we know that
$W_k$ is reduced and
\[
        n^+(W_k)=\binom{k}{2}+1,\qquad n^-(W_k)=k-1.
\]

Thus $2n^+(W_k)=k(k-1)+2$, whereas
$n^-(W_k)(n^-(W_k)+1)=(k-1)k$. 
Hence
\[
        2n^+(W_k)=n^-(W_k)(n^-(W_k)+1)+2,
\]
so the conjectured inequality fails for every \(k\ge5\).
\end{proof}

\begin{rem}
The same family also disproves Conjecture 5.1 of \cite{AkbariEtAl}, namely
\[
        2n\le (n-n^+(G))(n-n^+(G)+3).
\]
Indeed, for \(W_k\),
\[
        n=|V(W_k)|=\binom{k+1}{2},\qquad n-n^+(W_k)=k-1.
\]
The proposed inequality would give
\[
        k(k+1)\le (k-1)(k+2),
\]
which is false.
\end{rem}

\section{The case \texorpdfstring{\(n^-(G)=4\)}{n-=4}}

The smallest member of the family is \(W_5\). It has
$\operatorname{In}(W_5)=(11,0,4)$,
and therefore already disproves (\ref{e-1}). We now show that deleting the vertex \(a_1\) from \(W_5\) gives a reduced graph with inertia \((10,0,4)\), answering (\ref{e-2}).

Let \(H=W_5-a_1\); see Figure~\ref{fig:W5}, where the deleted vertex $a_1$ is marked
in red. Set \(X=\{2,3,4,5\}\).
The vertices of \(H\) are naturally divided into
\[
        A=\{a_i:i\in X\},\qquad
        U=\{b_{\{1,i\}}:i\in X\},\qquad
        V=\left\{b_S:S\in\binom X2\right\}.
\]

The outer ten vertices in Figure~\ref{fig:W5} induce the Petersen graph.
The sizes of these three sets are \(4,4,6\). This partition is equitable, and the corresponding quotient matrix is
\[
        Q=
        \begin{pmatrix}
        3&1&3\\
        1&0&3\\
        2&2&1
        \end{pmatrix}.
\]
Therefore the constant-on-cells subspace contributes
\begin{equation}\label{e-10}
        \det(xI-Q)=x^3-4x^2-10x+7.
\end{equation}

It remains to describe the nonconstant part. Let \(D\) be the \(4\times6\) incidence matrix between the points of \(X\) and the 2-subsets of \(X\). Let \(R\) be the adjacency matrix of the matching on \(\binom X2\) which sends a 2-subset to its complement in \(X\). Let \(E\) be the \(4\times6\) matrix defined by
\[
        E_{i,S}=1\quad\Longleftrightarrow\quad i\notin S.
\]
If \(\mb{x}\in\mathbb R^4\) satisfies \(\mb{x}\perp\mb{1}_4\), and
if $\mb{z}=D^T \mb{x}$, then
\begin{equation}\label{e-11}
        DD^T\mb{x}=2 \mb{x},\qquad E \mb{z}=-2\mb{x},\qquad E^T \mb{x}=- \mb{z},\qquad R \mb{z}=- \mb{z}.
\end{equation}
These identities follow from \(D D^T=2I_4+J_4\), from \(E=J_{4,6}-D\), and from the fact that the complement of a two-subset of \(X\) has incidence vector \(\mb{1}_4-\mb{1}_S\).

Let \(U_0=\mb{1}_4^\perp\). Since \(D^T\) is injective on \(U_0\), the space
\[
        \mathcal U=
        \{(\mb{\xi},\mb{\eta},D^T\mb{\zeta}):
        \mb{\xi},\mb{\eta},\mb{\zeta}\in U_0\}
\]
has dimension \(9\). Indeed, if \(D^T\mb{\zeta}=\mb{0}\) with
\(\mb{\zeta}\in U_0\), then
\[
        DD^T\mb{\zeta}=2\mb{\zeta}=\mb{0},
\]
and hence \(\mb{\zeta}=\mb{0}\).
For \(\mb{\xi},\mb{\eta},\mb{\zeta}\in U_0\),
using (\ref{e-11}), we obtain
\[
\begin{aligned}
A(H)(\mb{\xi},\mb{\eta},D^T\mb{\zeta})
&=
\bigl(
        -\mb{\xi}+\mb{\eta}+2\mb{\zeta},\,
        \mb{\xi}-2\mb{\zeta},\,
        D^T(\mb{\xi}-\mb{\eta}-\mb{\zeta})
\bigr).
\end{aligned}
\]
Hence \(\mathcal U\) is \(A(H)\)-invariant. Under the identification
\[
        (\mb{\xi},\mb{\eta},D^T\mb{\zeta})
        \longleftrightarrow
        (\mb{\xi},\mb{\eta},\mb{\zeta}),
\]
the induced action is given by the \(3\times3\) matrix
\[
        M=
        \begin{pmatrix}
        -1&1&2\\
        1&0&-2\\
        1&-1&-1
        \end{pmatrix}.
\]
Therefore this \(9\)-dimensional part contributes
\begin{equation}\label{e-12}
        \det(xI-M)^3
        =
        (x-1)^3(x^2+3x-1)^3.
\end{equation}

Finally, \(\ker D\) has dimension \(2\), because \(D D^T=2I_4+J_4\)
is nonsingular. If \(\mb{y}\in\ker D\), then \(\mb{1}_6^T\mb{y}=0\), hence \(E\mb{y}=0\).
The equations \(D\mb{y}=0\) also imply
\[
        y_S=y_{X\setminus S}
        \qquad
        \left(S\in\binom X2\right).
\]
Indeed, if \(S=\{i,j\}\) and \(X\setminus S=\{p,q\}\), then subtracting
the sum of the \(p,q\)-rows of \(D\mb{y}=\mb{0}\) from the sum of the
\(i,j\)-rows gives
\[
        2\bigl(y_S-y_{X\setminus S}\bigr)=0.
\]
Therefore \(R\mb{y}=\mb{y}\), and \(\ker D\) contributes an additional
factor \((x-1)^2\).

The three invariant parts described above have dimensions \(3\), \(9\), and
\(2\), respectively, and hence exhaust the \(14\)-dimensional space of \(H\).
Combining (\ref{e-10}), (\ref{e-12}), and the last factor gives
\begin{equation}\label{e-13}
\chi_H(x)=(x-1)^5(x^2+3x-1)^3(x^3-4x^2-10x+7).
\end{equation}
The quadratic \(x^2+3x-1\) has one positive root and one negative root. The cubic
\[
        p(x)=x^3-4x^2-10x+7
\]
has exactly one negative root: indeed, \(p(-3)<0<p(-2)\), while Descartes' rule applied to \(p(-x)\) gives at most one negative root. Since \(p\) is a factor of the characteristic polynomial of the real symmetric matrix \(A(H)\), all roots of \(p\) are real. As \(p(0)\ne0\), the other two roots are positive. It follows from (\ref{e-13}) that
\[
        \operatorname{In}(H)=(10,0,4).
\]

It remains to check that \(H\) is reduced. This is again immediate from the neighbourhood vectors. Vertices in \(A,U,V\) have degrees \(7,4,5\), respectively, so twins can occur only inside one of these three classes. For \(a_i,a_j\in A\), the \(U\)-part of their neighbourhood vectors is \(\mb{e}_i,\mb{e}_j\), and these are distinct when \(i\ne j\). For \(b_{\{1,i\}},b_{\{1,j\}}\in U\), the \(A\)-part of their neighbourhood vectors is again \(\mb{e}_i,\mb{e}_j\). For \(b_S,b_T\in V\), the \(A\)-part of their neighbourhood vectors is \(\mb{1}_S,\mb{1}_T\), which is distinct when \(S\ne T\). Thus \(H\) has no twins. Since all three degrees \(7,4,5\) are positive, \(H\) has no isolated vertices.

\begin{cor}
There exists a reduced graph $G$ with
\[
        n^-(G)=4,\qquad n^+(G)=10.
\]
\end{cor}

\begin{proof}
Take \(G=H=W_5-a_1\).
\end{proof}

\section{Conclusion}

The graphs \(W_k\) with \(k\ge5\) form an infinite family of reduced counterexamples to the proposed inequality
\[
        2n^+(G)\le n^-(G)(n^-(G)+1).
\]
They satisfy
\[
        \operatorname{In}(W_k)=\left(\binom{k}{2}+1,\ 0,\ k-1\right),
\]
so the conjectured upper bound is exceeded by exactly \(2\). The
smallest member \(W_5\) has inertia \((11,0,4)\), and the graph
\(W_5-a_1\) is reduced with inertia \((10,0,4)\).

\vspace{6mm}

\n{\bf Acknowledgements:} %The authors would like to thank the
%anonymous referees for their constructive corrections and valuable
%comments on this paper, which have considerably improved the
%presentation of this paper.
This work was supported by the Fujian Key Laboratory of Granular
Computing and Applications.

\paragraph{Data availability.}
Data sharing is not applicable to this article as no datasets were generated or
analysed during the current study.

\paragraph{Conflict of interest.}
The authors have no relevant financial or non-financial interests to disclose.

\end{document}